\documentclass{article}
\usepackage{fullpage}
\usepackage{amsmath}
\usepackage{amssymb}
\usepackage{orcidlink}
\usepackage{amsthm}
\usepackage{enumitem}
\usepackage{hyperref}
\usepackage{comment}

\begin{document}

\newtheorem{theorem}{Theorem}
\newtheorem{lemma}{Lemma}
\newtheorem{corollary}{Corollary}
\newtheorem{example}{Example}
\newtheorem{proposition}{Proposition}
\theoremstyle{remark}
\newtheorem{remark}{Remark}
\theoremstyle{definition}
\newtheorem*{definition}{Definition}
\newtheorem{assumption}{Assumption}

\title{The Asymptotic Equivalence of Level-Based and Share-Based Loss Functions\footnote{This paper is a revision of general results of research undertaken by Census Bureau Staff. The views expressed are attributable to the author and do not necessarily reflect those of the Census Bureau. }}
\author{Charles D. Coleman\orcidlink{0000-0001-6940-8117}\thanks{ 
Timely Analytics, LLC, 
E-mail: info@timely-analytics.com}
}

\maketitle

\abstract{Level-based and share-based loss functions are asymptotically equivalent if, in the limit, their averages converge almost surely to a constant ratio.  These loss functions take a target value and its realization as arguments and are often used to measure accuracy.  The equivalence is proved for a large class of loss functions, the weighted exponentiated functions, when the weights are decomposable as a particular product form.  An upshot is that when losses are averaged for a large number of units, differences in ratios and, hence, ranks, are negligible, when the average (or summed) difference between the target values and their realizations is around zero.  This implies the almost sure asymptotic convergence of numerical and distributive accuracy when using these loss functions.}

\section{Introduction}

Given a set of target values and their realizations, loss functions can be constructed to measure their accuracy.  The lower the total loss a set of realizations has, the more accurate they are considered.  Accuracy can either be numeric, comparing levels, or distributive, comparing shares.  The same loss functions can be used to measure either, with the arguments being either levels or shares, depending on the problem.  This paper shows that the means of level-based and share-based loss functions are asymptotically equivalent,\footnote{In terms of Cohen and Zhang (1988), these are the “allocative” and “proportional” loss functions.} for a broad class of loss functions, the weighted exponentiated loss functions.  That is, as the number of units measured increases, the measures using a member of this classconverge almost surely to a fixed ratio.  The upshot of this is that numeric and distributive accuracy almost surely converge when using these loss functions.  Thus, arguments as to whether one or the other is to be optimized are asymptotically moot.  Another way of viewing this is that the ratios of mean level-based and share-based loss functions for different sets of realizations almost surely converge.  Thus, using levels or shares as the arguments has no effect asymptotically on the rankings of different sets of realizations.

The target and realized values can have several interpretations.  The most obvious interpretation is that the target value is indeed a target and some process produces a possibly different realization.  This is the case in taking a census: the target is the true population and the realization is the census count.  The realizations can be simulated populations and the targets the “true” populations.  This type of simulation is often used  to estimate the risk, that is, the expected loss of population adjustments.\footnote{Belin and Rolph (1994, p. 491) have a lengthy, nonexhaustive bibliography of works using loss functions in the context of census adjustment.  Spencer (1980) appears to have created the idea of estimating risk by simulation  in the related context of general revenue sharing.} The target can also be the actual value and the realization is a prediction of the target.  Other interpretations are possible.

One implication of this work is that the index of dissimilarity is asymptotically equivalent to the mean and total absolute differences.  In fact, while the index of dissimilarity is used as a type of accuracy measure, it is really meant to measure segregation.\footnote{See White (1986) for a discussion of the index of dissimilarity and other measures of segregation.}  Both dissimilarity and segregation are ill-defined concepts.\footnote{White (1986, p. 199) defines a segregation statistic as “a single number that characterizes the two-dimensional distribution of the population’s subgroups across units...”  This is a vague statement, as it could refer to concepts other than segregation.  Before this sentence, he states that a segregation statistic “provide[s] an implicit definition of segregation.”  Thus, segregation is defined relative to a statistic and not the other way around.  Dissimilarity (White, 1986, pp. 202-203) is also defined in this manner.}.  The asymptotic equivalence of the index of dissimilarity and the mean and total absolute differences eliminates the first’s usefulness as an accuracy measure in most cases.\footnote{Armstrong (1985, p.347) states that the mean absolute “deviation” (i.e., difference) is appropriate when the cost function is linear.  That is, the penalty associated with an error is proportionate to its absolute value.}  The index of dissimilarity can be used to test goodness-of-fit in categorical data models (Kuha and Firth, 2011).

Section \ref{aseq} defines asymptotic equivalence and ratio asymptotic equivalence.  Although ratio asymptotic equivalence is the result of interest, it is equivalent to a differential definition of asymptotic equivalence, which we use in the proofs.  Global asymptotic equivalence is defined as the term-wise convergence of all summands used in the definition of asymptotic equivalence.  Section \ref{wedlf} defines the weighted exponentiated loss functions, proves asymptotic equivalence and provides information on the behavior of the ratio of the means of their level-based and share-based versions.  Section \ref{aeexamples} provides some examples of common, asymptotically equivalent accuracy measures.   Section \ref{small_sample} demonstrates by example how asymptotic equivalence does not imply small sample equivalence.  Section \ref{empirical} calculates asymptotically equivalent measures for 1990 census county populations and two sets of population estimates, demonstrating partial convergence for a large number of units.  Section \ref{conclusion} concludes this paper.

\section{Asymptotic Equivalence\label{aseq}}

Let $X_i, Y_i \ge 0$ be the levels and $x_i = X_i/\sum_{j=1}^n X_j$  and $y_i = Y_i/\sum_{j=1}^n Y_j$  be their respective shares.  Let $X = \{X_i\}_{i=1}^n$, $Y = \{Y_i\}_{i=1}^n,x = \{x_i\}_{i=1}^n$ and $y = \{y_i\}_{i=1}^n$ be the sets of levels and shares. Let $L = L(X_i, Y_i)$ be the loss function using level arguments and $\ell = \ell(x_i, y_i)$ be the same loss function using share arguments.  Let $\bar{L}_n = n^{-1}\sum_{i=1}^n L(X_i,Y_i)$ and $\bar{\ell}_n = n^{-1}\sum_{i=1}^n \ell(X_i,Y_i)$ be the average level- and share-based average losses, respectively. Then, we define asymptotic equivalence and ratio asymptotic equivalence as:
\begin{definition}[Asymptotic Equivalence]
 $\lim_{n \rightarrow\infty}(\bar{L}_n - K \bar{\ell}_n) = 0$   a.s. for some constant $K > 0$.
\end{definition}
\begin{definition}[Ratio Asymptotic Equivalence]
$\lim_{n \rightarrow\infty} \frac{\bar{\ell}_n}{\bar{L}_n} = K$  a.s. for some constant $K > 0$.
\end{definition}
\noindent
The reader can easily verify that these two definitions are equivalent.

	Asymptotic equivalence has global variants.  Global asymptotic equivalence requires the almost sure term-wise convergence of each individual loss function, which strong global asymptotic strengthens to always.
\begin{definition}[Global Asymptotic Equivalence]
$\lim_{n \rightarrow\infty}(L_i - K \ell_i) = 0$ a.s. for all $i$ and some constant $K > 0$.
\end{definition}
\begin{definition}[Strong Global Asymptotic Equivalence]
$\lim_{n \rightarrow\infty}(L_i - K \ell_i) = 0$ for all $i$ and some constant $K > 0$.
\end{definition}
\noindent
Global and strong global asymptotic equivalence can be obtained by imposing convergence restrictions on the differences.  They are stronger than asymptotic equivalence, as the limit applies to each difference.  Moreover, strong global asymptotic convergence always occurs by definition.

\section{The Weighted Exponentiated Difference Loss Function\label{wedlf}}

The weighted exponentiated difference loss function is of the form $L(X_i,Y_i)=|X_i-Y_i|^p w(X_i)$  or $L(X_i,Y_i)=|X_i-Y_i|^p w(Y_i)$  where $0 < p < \infty$ and $w(X_i)$, $0 < w(Xi) \le 1$ for all nonnegative arguments, $w(Y_i$) has a similar factorization and  $S_{xn} = \sum_{i=1}^n X_i$ and $S_{yn}$ is similarly defined.  Let $X_i$ be the realized and $Y_i$ the target values, respectively. Examples of these loss functions include the absolute difference,  $|X_i-Y_i|$, the squared difference, $(X_i-Y_i)^2$,  the absolute percentage difference,  $|X_i-Y_i|/Y_i$,  the Webster/Saint Lagüe or $\chi^2$ loss function, $(X_i-Y_i)^2/Y_i$, the Huntington-Hill/Method of Equal Proportions loss function, $(X_i-Y_i)^2/X_i$,  and the Cobb-Douglas loss function $|X_i-Y_i|^p Y_i^q$, $p > 0, q<0$ (Coleman, 2025) of which the preceding are special cases with $q \le 0$.

\subsection{Proof of Asymptotic Equivalence}

Aymptotic equivalence is really a special case of a Cesàro mean, the  limit of the sequence of arithmetic means of partial sums when that limit exists.\footnote{The derivations in this and the next Subsection were mainly done by ChatGPT.} Table  \ref{tab:cesaro_assumptions} summarizes each condition, its control or guarantee, and interpretation. Conditions (\ref{con1})--(\ref{con3}) ensure stability and boundedness, (\ref{con4}) enforces scaling regularity, and (\ref{con5}) guarantees asymptotic equivalence of the level and share domains. Condition (\ref{con5}) is one of many sufficient conditions to realize a less intuitive necessary condition. This choice allows a finite number of differences to remain bounded away from 0, as can happen when a method produces bad realizations for a set of targets. These are dominated by reasonable ones. When the targets span a few orders of magnitude, such as populations of geographic areas defined without regard to population, most will be small.  Even when targets are defined relative to population, such as city populations, the targets usually skew small.  Theorem~\ref{thm:cesaro_mean_ratio_full} proves that the level-based and share-based forms of the loss function form a weighted Cesàro mean. Asymptotic equivalence is then Corollary~\ref{cor:asymptotic_equivalence}.

Exchangeability, the ability to permute a finite number of variables without altering their joint distribution, is a natural assumption to make in this context.  It leads to a simpler proof of Theorem~\ref{thm:cesaro_mean_ratio_full}.  Remark \ref{rem:exchangeability} describes how the simplification occurs.


\begin{assumption}[Regularity conditions for weighted Cesàro mean convergence]
\label{ass:cesaro_conditions}
Let $(X_i,Y_i)_{i\ge1}$ be a sequence of nonnegative random pairs.
The following conditions are assumed throughout
Theorem~\ref{thm:cesaro_mean_ratio_full} and its corollaries.

\begin{enumerate}[label=(A\arabic*), ref=A\arabic*]
\item \textbf{Finite moments.}    \label{con1} 
There exists $p>0$ such that
\[
\mathbb{E}|X_i|^p<\infty, \qquad
\mathbb{E}|Y_i|^p<\infty,
\quad\text{for all }i\ge1.
\]

\item \textbf{Weighted Ces\`aro boundedness.}    \label{con2} 
The deterministic weighting function $w(\cdot)>0$ on $(0,1]$
satisfies
\[
\frac{1}{n}\sum_{i=1}^n w(x_i)\le M,
\qquad
\frac{1}{n}\sum_{i=1}^n
w(x_i)\big(|X_i|^p+|Y_i|^p\big)=O(1),
\]
and the family $\{\,w(x_i)\big(|X_i|^p+|Y_i|^p\big):i\ge1\,\}$ is
\emph{uniformly Ces\`aro-integrable}; that is, every subset average
\[
\frac{1}{|A_n|}\sum_{i\in A_n}
w(x_i)\big(|X_i|^p+|Y_i|^p\big)
\]
is bounded by the same deterministic envelope as the full-sample
average whenever $|A_n|/n$ remains bounded away from zero.\footnote{
In practice,  uniform Ces\`aro-integrability is automatically 
satisfied whenever the weighting function $w$ is bounded on $(0,1]$ and the 
sequence $(X_i,Y_i)$ obeys a weak law of large numbers (e.g.,\ independence, 
strong mixing, or exchangeability) with finite $2p$-moments, which ensures 
that subset averages cannot diverge faster than full-sample averages.
\label{cesaro}}

\item \textbf{Stable total means.}    \label{con3} 
The partial sums satisfy, almost surely,
\[
\frac{S_{xn}}{n}\to\mu_x>0,
\qquad
\frac{S_{yn}}{n}\to\mu_y>0,
\]
with $S_{yn}=\sum_{j=1}^n Y_j$.
These limits define the deterministic scale constants $\mu_x,\mu_y$
appearing in the asymptotic constant $K=\mu^p$.

\item \textbf{Asymptotic regularity of weights.}    \label{con4} 
The weighting function $w(\cdot)$ is measurable and bounded on $(0,1]$
and admits Cesàro limits
\[
\frac{1}{n}\sum_{i=1}^n w(x_i) \to \bar{w}_x,
\qquad
\frac{1}{n}\sum_{i=1}^n w(y_i) \to \bar{w}_y,
\]
for finite constants $\bar{w}_x,\bar{w}_y>0$.

\item \textbf{Sparse deviations.}    \label{con5} 
There exists $\varepsilon_n\downarrow0$ and
\[
B_n=\{\,i\le n:\ |Y_i-(S_{xn}/S_{yn})X_i|>\varepsilon_n\,\}
\]
such that $|B_n|=o(n)$.

\end{enumerate}
\end{assumption}

\begin{table}[h!]
\centering
\caption{Roles of the regularity conditions (\ref{con1})--(\ref{con5}) in
Assumption~\ref{ass:cesaro_conditions}}
\label{tab:assumptions}
\renewcommand{\arraystretch}{1.25}
\setlength{\tabcolsep}{4pt}
\begin{tabular}{|c|p{3.2cm}|p{4.5cm}|p{4.7cm}|}
\hline
\textbf{Label} & \textbf{Condition} &
\textbf{Controls / Guarantees} & \textbf{Interpretation} \\ 
\hline
(\ref{con1}) & Finite $p$th moments &
Integrability of $|X_i|^p, |Y_i|^p$ &
Prevents divergence of averages; enables law-of-large-numbers arguments. \\

\hline
(\ref{con2}) & Weighted Ces\`aro boundedness and uniform integrability &
Uniform control of weighted magnitudes on the full sample and on large subsets &
Ensures weights do not dominate the mean; keeps averages $O(1)$ and guarantees that restricting to subsets (e.g.\ sparse or complement sets) cannot create explosive contributions. \\

\hline
(\ref{con3}) & Stable total means &
Linear growth $S_{xn}/n\!\to\!\mu_x$, $S_{yn}/n\!\to\!\mu_y$ &
Provides deterministic scale constants $\mu_x,\mu_y$ to determine $K$. \\

\hline
(\ref{con4}) & Asymptotic regularity of weights &
Existence of Ces\`aro limits of $w(x_i)$, $w(y_i)$ &
Guarantees that the weighting function behaves stably under averaging. \\

\hline
(\ref{con5}) & Sparse deviations &
Vanishing fraction $|B_n|/n\!\to\!0$ where
$|Y_i-(S_{xn}/S_{yn})X_i|>\varepsilon_n$ &
Ensures only a negligible subset of indices violate proportional mean convergence; drives asymptotic equivalence. \\
\hline
\end{tabular}
\label{tab:cesaro_assumptions}
\end{table}
Table \ref{tab:assumptions} summarizes Assumption~\ref{ass:cesaro_conditions}'s conditions, controls/guarantees and interpretations. Conditions (\ref{con1})--(\ref{con3}) establish stability and boundedness, (\ref{con4}) ensures
regularity of the weighting function, and (\ref{con5}) enforces sparse deviations in
$(X_i,Y_i)$ sequences.

Lemma \ref{lem:level_share_identity} proves a simple mathematical identity that is subsequently used in Theorem~\ref{thm:cesaro_mean_ratio_full}.
\begin{lemma}[Identity linking level and share differences]
\label{lem:level_share_identity}
Let $c_n=S_{xn}/S_{yn}$.  Since $S_{xn}/S_{yn} >0$, for every $i$ and any
$p>0$,
\[
|x_i-y_i|^p
=\frac{1}{S_{xn}^p}\,\big|X_i-c_nY_i\big|^p.
\]
\end{lemma}

\begin{proof}
By direct substitution,
\[
x_i-y_i
=\frac{X_i}{S_{xn}}-\frac{Y_i}{S_{yn}}
=\frac{S_{yn}X_i-S_{xn}Y_i}{S_{xn}S_{yn}}
=\frac{1}{S_{xn}}\Big(X_i-\frac{S_{xn}}{S_{yn}}Y_i\Big)
=\frac{1}{S_{xn}}(X_i-c_nY_i).
\]
Taking absolute values and raising to the $p$th power yields the stated
identity.
\end{proof}

\begin{theorem}[Weighted Cesàro mean convergence under sparse deviations]
\label{thm:cesaro_mean_ratio_full}
Let $(X_i,Y_i)_{i\ge1}$ satisfy
Assumption~\ref{ass:cesaro_conditions}.
Assume
\[
\frac{S_{xn}}{n}\to \mu_x>0, \qquad
\frac{S_{yn}}{n}\to \mu_y>0.
\]
If the sparse deviation condition (\ref{con5}) holds, then for any $p>0$,
\[
\frac{1}{n}\sum_{i=1}^n |X_i-Y_i|^p\,w(X_i)
\;-\;
K\,\frac{1}{n}\sum_{i=1}^n |x_i-y_i|^p\,w(x_i)
\xrightarrow[]{a.s.} 0,
\]
where $K=\mu_x^p$ or $\mu_y^p$ according to whether the weights are taken
on $X$ or $Y$.
\end{theorem}

\begin{proof}
We give the argument for $X$-weights; the $Y$-weighted case is symmetric.
Let $c_n:=S_{xn}/S_{yn}$, so $c_n\to \mu_x/\mu_y\in(0,\infty)$ a.s. by (\ref{con3}).

\medskip
\noindent\textbf{Step 1 (Level–share identity).}
By Lemma~\ref{lem:level_share_identity},
\[
|x_i-y_i|^p=\frac{1}{S_{xn}^p}\,|X_i-c_nY_i|^p.
\]
Hence,
\[
\frac{1}{n}\sum_{i=1}^n w(x_i)|x_i-y_i|^p
=\frac{1}{S_{xn}^p}\,\frac{1}{n}\sum_{i=1}^n w(x_i)|X_i-c_nY_i|^p.
\]
Therefore, it suffices to prove
\begin{equation}\label{eq:keydiff}
\frac{1}{n}\sum_{i=1}^n w(x_i)\Big(|X_i-Y_i|^p-|X_i-c_nY_i|^p\Big)\xrightarrow[]{a.s.}0,
\end{equation}
because
\[
\frac{1}{n}\sum |X_i-Y_i|^p w(X_i)
=\Big(\frac{S_{xn}}{n}\Big)^{\!p}\,\frac{1}{n}\sum w(x_i)|x_i-y_i|^p + o(1),
\]
and $(S_{xn}/n)^p\to \mu_x^p=K$.

\medskip
\noindent\textbf{Step 2 (Sparse–nonsparse decomposition).}
Let $\varepsilon_n\downarrow0$ and $B_n=\{i\le n:\ |Y_i-c_nX_i|>\varepsilon_n\}$,
so $|B_n|=o(n)$ by (\ref{con5}); set $G_n=\{1,\dots,n\}\setminus B_n$.
Define
\[
\Delta_{i,n}:=|X_i-Y_i|^p-|X_i-c_nY_i|^p,\qquad
D_n:=\frac{1}{n}\sum_{i=1}^n w(x_i)\Delta_{i,n}
= D_n^{(G)}+D_n^{(B)}.
\]

\medskip
\noindent\textbf{Step 3 (Control on $G_n$ for all $p>0$).}
Fix $i\in G_n$ and set $\Delta_{i,n}=|X_i-Y_i|^p-|X_i-c_nY_i|^p$.

\emph{Case i: $p>1$.}
For the differentiable map $f(t)=|X_i-tY_i|^p$, the mean–value theorem yields
\[
|\Delta_{i,n}|=|f(1)-f(c_n)|
   \le |c_n-1|\,\sup_{t\in[1,c_n]}|f'(t)|
   \le p\,|c_n-1|\,|Y_i|\,\sup_{t\in[1,c_n]}|X_i-tY_i|^{p-1}.
\]
Since $|X_i-tY_i|\le |X_i|+|t|\,|Y_i|$ and $|t|\le C_t<\infty$ for large $n$
(because $c_n\to\mu_x/\mu_y$), we have
\[
|\Delta_{i,n}|
\le C_p\,|c_n-1|\,|Y_i|\,(|X_i|^{p-1}+|Y_i|^{p-1}).
\]
Averaging and applying (\ref{con2}),
\[
|D_n^{(G)}|
\le C_p\,|c_n-1|\,
\frac{1}{n}\sum_{i=1}^n w(x_i)\,(|X_i|^p+|Y_i|^p)
=O(|c_n-1|).
\]

\emph{Case ii: $0<p\le1$.}
Since $u\mapsto|u|^p$ is $p$–Hölder and subadditive,
\[
|\Delta_{i,n}|
=\big||X_i-Y_i|^p-|X_i-c_nY_i|^p\big|
\le |(c_n-1)Y_i|^p
\le C_p\,|c_n-1|^p\,|Y_i|^p,
\]
so
\[
|D_n^{(G)}|
\le C_p\,|c_n-1|^p\,
\frac{1}{n}\sum_{i=1}^n w(x_i)\,|Y_i|^p
=O(|c_n-1|^p),
\]
again by (\ref{con2}).

\smallskip
\noindent In both cases the bound is uniform in $n$ and can be summarized as
\[
|D_n^{(G)}| = O\!\big(|c_n-1|^{\min(p,1)}\big),
\]
which vanishes when $\mu_x=\mu_y$ (that is, $c_n\to1$).

\medskip
\noindent\textbf{Step 4 (Control on $B_n$).}
For all $i$ and any $p>0$,
$|\Delta_{i,n}|\le |X_i-Y_i|^p+|X_i-c_nY_i|^p\le C_p(|X_i|^p+|Y_i|^p)$.
Therefore,
\[
|D_n^{(B)}|
\le \frac{|B_n|}{n}\cdot
\frac{1}{|B_n|}\sum_{i\in B_n} C_p\,w(x_i)\,(|X_i|^p+|Y_i|^p)
= O\!\Big(\frac{|B_n|}{n}\Big)\to 0,
\]
by (\ref{con2}) and (\ref{con5}).

\medskip
\noindent\textbf{Step 5 (Conclusion and constant identification).}
Combining Steps 3–4, $D_n\to 0$ a.s., i.e., equation \eqref{eq:keydiff} holds.
Using Step 1 and $S_{xn}/n\to\mu_x>0$,
\[
\frac{1}{n}\sum_{i=1}^n |X_i-Y_i|^p\,w(X_i)
-\Big(\frac{S_{xn}}{n}\Big)^{\!p}
\frac{1}{n}\sum_{i=1}^n w(x_i)|x_i-y_i|^p \xrightarrow[]{a.s.}0,
\]
and $(S_{xn}/n)^p\to \mu_x^p=K$ gives the claim. The $Y$-weighted case is identical with
$K=\mu_y^p$.
\end{proof}

Since asymptotic equivalence is based on using the same $w()$ regardless of its argument, it follows as the simple Corollary~\ref{cor:asymptotic_equivalence}.
\begin{corollary}[Asymptotic equivalence]
\label{cor:asymptotic_equivalence}
The result of
Theorem~\ref{thm:cesaro_mean_ratio_full} remains valid when
$w(X_i)=w(Y_i)=w(\cdot)$.
Under the bounded-moment condition
\[
\frac{1}{n}\sum_{i=1}^n w(x_i)\big(|X_i|^p+|Y_i|^p\big)=O(1),
\]
both totals satisfy $S_{xn}/n\to\mu_x>0$ and $S_{yn}/n\to\mu_y>0$, and
the same constant $K=\mu^p$ links the level- and share-weighted Cesàro
means ($\mu=\mu_x$ for $X$-weights, $\mu=\mu_y$ for $Y$-weights$)$.
\end{corollary}

Remark \ref{rem:explicit_contants} gives $K$ for the six loss functions described at the beginning of this Section.
\begin{remark}[Explicit constants for standard weighted exponentiated difference losses]
\label{rem:explicit_contants}
We list the values of $K$ and $p$ and the form of $w()$ for the six loss functions mentioned above. $K=\mu_x^p$ gives the deterministic
asymptotic scale linking the level-weighted and share-weighted Cesàro
means of the loss, with $\mu=\mu_x$ or $\mu_y$ according to whether the
weights are taken on $X$ or on $Y$.

\begin{enumerate}
\item \textbf{Absolute-difference loss:}
\[
L(X,Y)=|X-Y|,\qquad w(x)\equiv1,\quad p=1,
\quad\Rightarrow\quad K=\mu_x.
\]

\item \textbf{Squared-difference (quadratic) loss:}
\[
L(X,Y)=|X-Y|^2,\qquad w(x)\equiv1,\quad p=2,
\quad\Rightarrow\quad K=\mu_x^2.
\]

\item \textbf{Absolute percentage error (APE):}
\[
L(X,Y)=\frac{|X-Y|}{Y},\qquad
w(y)=\frac{1}{y},\quad p=1,
\quad\Rightarrow\quad K=1.
\]

\item \textbf{Webster/Saint Lagüe ($\chi^2$) loss:}
\[
L(X,Y)=\frac{|X-Y|^2}{Y},\qquad
w(y)=\frac{1}{y},\quad p=2,
\quad\Rightarrow\quad K=\mu_y.
\]

\item \textbf{Huntington--Hill/Method of Equal Proportions loss:}
\[
L(X,Y)=\frac{|X-Y|^2}{X},\qquad
w(x)=\frac{1}{x},\quad p=2,
\quad\Rightarrow\quad K=\mu_x.
\]

\item \textbf{Cobb-Douglas loss:}
\[
L(X,Y)=|X-Y|^pY^q,\qquad
w(y)=y^q, q<0,\quad p>0,
\quad\Rightarrow\quad K=\mu_y.
\]
\end{enumerate}

\end{remark}

Finally, Remark \ref{rem:exchangeability} shows how exchangeability greatly simplifies the proof of Theorem~\ref{thm:cesaro_mean_ratio_full}.  Note that independence is not assumed.
\begin{remark}[On exchangeability]
\label{rem:exchangeability}
From footnote \ref{cesaro}, if (\ref{con1}) holds and $(X_i,Y_i)$ is an exchangeable sequence with finite
$p$-th moments, then these limits and bounds follow automatically from
the strong law of large numbers for exchangeable arrays
(de~Finetti representation). Thus, exchangeability provides another natural
probabilistic mechanism ensuring the Theorem's conclusion.
\end{remark}

Many other assumptions on the behavior of the $X_i$ and $Y_i$ generate other versions of Theorem~\ref{thm:cesaro_mean_ratio_full}. Obviously, showing all of them would require generating many more proofs for little gain in understanding. Moreover, many of these assumptions are unintuitive or inappropriate for our context. Rather than deriving more proofs, we settle for the two assumption sets for which we have proofs.

\subsection{Convergence Rate}

We give the convergence rate of $c_n$ under exchangeability and sparse deviations in Proposition \ref{prop:cn_speed}. The rate is given for the combination.  If only one of the conditions apply, then only the applicable part of the final convergence rate applies.
\begin{proposition}[Convergence rate of $c_n$ under exchangeability and sparse deviations]
\label{prop:cn_speed}
Let $(X_i,Y_i)_{i\ge1}$ satisfy Assumptions \textup{(A1)}--\textup{(A5)} and
suppose the sequence is exchangeable. Define
$S_{xn}=\sum_{i=1}^n X_i$, $S_{yn}=\sum_{i=1}^n Y_i$, and
$c_n=S_{xn}/S_{yn}$. Then
\[
c_n - \frac{\mu_x}{\mu_y}
=
O_p(n^{-1/2})
\;+\;
O(\varepsilon_n)
\;+\;
O_p\!\left(\frac{|B_n|}{n}\right).
\]
The stochastic term $O_p(n^{-1/2})$ arises from exchangeability,
the deterministic alignment term $O(\varepsilon_n)$ arises from
the deviation bound on $G_n$, and the sparse-deviation term
$O_p(|B_n|/n)$ arises from the negligibility of $B_n$.
\end{proposition}

\begin{proof}
Write
\[
c_n - \frac{\mu_x}{\mu_y}
= 
\frac{S_{xn}}{S_{yn}} - \frac{\mu_x}{\mu_y}.
\]
Since the sequence is exchangeable, de Finetti's theorem implies that
conditionally on a latent variable $\Theta$, the pairs $(X_i,Y_i)$ are
i.i.d. with finite second moments by (A1). Hence, by the conditional
central limit theorem,
\[
\frac{S_{xn}}{n}
=
\mu_x(\Theta) + O_p(n^{-1/2}),
\qquad
\frac{S_{yn}}{n}
=
\mu_y(\Theta) + O_p(n^{-1/2}),
\]
where $\mu_x(\Theta)=\mathbb{E}[X_1\mid\Theta]$,
$\mu_y(\Theta)=\mathbb{E}[Y_1\mid\Theta]$.
A standard expansion for ratios of sample means then gives
\[
\frac{S_{xn}}{S_{yn}}
=
\frac{\mu_x(\Theta)}{\mu_y(\Theta)}
+ O_p(n^{-1/2}).
\]
If the directing measure is degenerate or the limits in (A3) are deterministic,
then $\mu_x(\Theta)\equiv\mu_x$ and $\mu_y(\Theta)\equiv\mu_y$, otherwise the
convergence is to the random limit $\mu_x(\Theta)/\mu_y(\Theta)$. In either
case the stochastic fluctuation satisfies
\[
c_n - \frac{\mu_x}{\mu_y}=O_p(n^{-1/2}).
\]

Next, we incorporate the sparse deviation structure (A5). Write
\[
Y_i = c_n X_i + \delta_{i,n}, \qquad
\delta_{i,n}=Y_i - c_n X_i.
\]
Summing over $i\le n$ obtains
\[
S_{yn}
=
c_n S_{xn}
+ \sum_{i\in G_n}\delta_{i,n}
+ \sum_{i\in B_n}\delta_{i,n}.
\]
Rearranging yields the exact identity
\begin{equation}\label{eq:cn_decomp}
c_n
=
\frac{S_{yn}}{S_{xn}}
- \frac{1}{S_{xn}}\sum_{i\in G_n}\delta_{i,n}
- \frac{1}{S_{xn}}\sum_{i\in B_n}\delta_{i,n}.
\end{equation}

On the set $G_n$, Assumption (A5) gives $|\delta_{i,n}|\le\varepsilon_n$, so
\[
\bigg|\frac{1}{S_{xn}}\sum_{i\in G_n}\delta_{i,n}\bigg|
\le
\frac{n\,\varepsilon_n}{S_{xn}}
=
O(\varepsilon_n)
\]
since $S_{xn}/n\to\mu_x>0$ by (A3).

On the sparse set $B_n$, we only use $|B_n|=o(n)$. Using
$S_{xn}=n\mu_x+o_p(n)$ and $|\delta_{i,n}|\le C(|X_i|+|Y_i|)$ with bounded
$p$-moments by (A1), we obtain
\[
\bigg|\frac{1}{S_{xn}}\sum_{i\in B_n}\delta_{i,n}\bigg|
\le
\frac{|B_n|}{S_{xn}}
\max_{i\in B_n}|\delta_{i,n}|
=
O_p\!\left(\frac{|B_n|}{n}\right).
\]

Substituting these two bounds into \eqref{eq:cn_decomp} together with the
root-$n$ expansion for $S_{yn}/S_{xn}$ completes the proof:
\[
c_n - \frac{\mu_x}{\mu_y}
=
O_p(n^{-1/2})
+ O(\varepsilon_n)
+ O_p\!\left(\frac{|B_n|}{n}\right).
\]
\end{proof}

\section{Examples of Asymptotic Equivalence\label{aeexamples}}

\begin{example}\label{ID}
The index of dissimilarity (ID), $\frac{1}{2n}\sum_{i=1}^n |x_i - y_i|$, and its multiple, the Total Absolute Error of Shares (Census, 2023), $\sum_{i=1}^n |x_i - y_i|$,   are asymptotically equivalent to the mean absolute difference,  $\frac{1}{n}\sum_{i=1}^n |X_i - Y_i|$. 
\end{example}
Example \ref{ID} shows a major limitation of the index of dissimilarity: it is asymptotically equivalent to the Total Absolute Difference, a measure with limited usefulness.  In particular, it is completely insensitive to the size of a unit. A large error in a large unit is considered the same as the same error in a smaller unit, even though the latter has a higher relative error.\footnote{It easy to prove the same of the index of dissimilarity. However, it has wider acceptance as an error measure.}
\begin{example}
The level-based Cobb-Douglas total loss function, $\sum_{i=1}^n |X_i-Y_i|^pY_i^q$,   is asymptotically equivalent to the share-based Cobb-Douglas total loss function,  $\sum_{i=1}^n |x_i-y_i|^py_i^q$, where $p > 0$ and $q < 0$.  In particular, the Webster/Saint Lagüe level loss function, $\chi^2 = \sum_{i=1}^n |X_i-Y_i|^2 Y_i^{-1}$, is asymptotically equivalent to  $\sum_{i=1}^n |x_i-y_i|^2 y_i^{-1}$, Pearson's $\chi^2$ divergence. Multiplying the latter by $n$ yields the $\chi^2$ statistic for shares.
\end{example}

\section{Asymptotic Equivalence Does Not Imply Small-Sample Equivalence\label{small_sample}}

This is best demonstrated by example.  Consider the data in Table \ref{tab:small_sample}.  The $x_{1i}$ and $x_{2i}$ are two different realizations of the $y_i$.
\begin{table}[]
\caption{Small Sample Example}
\centering
\label{tab:small_sample}
\begin{tabular}{lrr}
\\
                          & Unit 1 & Unit 2 \\ \hline
$y_i$                        & 10     & 990    \\
Share of Total            & 0.01   & 0.99   \\ \hline
$x_{i1}$                       & 11     & 999    \\
Share of Total            & 0.0109 & 0.9891 \\
Absolute Difference       & 1      & 9      \\
Absolute Share Difference & 0.0009 & 0.0009 \\ \hline
$x_{i2}$                        & 15     & 995    \\
Share of Total            & 0.0149 & 0.9851 \\
Absolute Difference       & 5      & 5      \\
Absolute Share Difference & 0.0049 & 0.0049 \\
\end{tabular}
\end{table}
Table \ref{tab:summary_statistics} shows the total absolute differences and indices of dissimilarity for the two sets of data in Table \ref{tab:small_sample}.  “Set $j$” refers to the $x_{ij}$.  In this example, $x_{11} + x_{21} = x_{12} + x_{22} = 1010$, so no scale effects occur from having different $S_{xn}$.
\begin{table}[]
\caption{Small Sample Summary Statistics}
\centering
\label{tab:summary_statistics}
\begin{tabular}{lrr}
\\
	& Set 1	& Set 2 \\
Total Absolute Difference	& 10	& 10 \\
Index of Dissimilarity	& 0.0004	& 0.0024
\end{tabular}
\end{table}
According to the index of dissimilarity, Set 1 is closer to the $y_i$ than Set 2, while the Total Absolute Difference is the same for both.  This example thus shows that two asymptotically equivalent measures are not necessarily equivalent in small samples.\footnote{Cohen and Zhang (1988) obtain a similar result, also for two units.  They find, unsurprisingly, that the level curves of two functions not related by a monotonic transformation,  are indeed different.}

\section{An Empirical Example\label{empirical}}

This example is based on the data in Davis (1994).  The Census Bureau produced population estimates for 3,141 counties for April 1, 1990.  These are the $X_{i1}$.  The 1990 census actual values are the $Y_i$.  The 1980 census actual values are the $X_{i2}$.  The measures considered are total absolute difference,  index of dissimilarity, $\chi^2$ and Fisher's $\chi^2$ divergence.  The idea is to compare the accuracy of estimating county populations by simply using the 1980 census values against the estimates produced before the 1990 census.  

Table \ref{tab:measures} displays the values of the measures and their ratios.
\begin{table}[]
\caption{Comparison of Accuracy Measures for 1990}
\label{tab:measures}
\centering
\begin{tabular}{lrrr}
\\
Measure                     & 1990 Estimates & 1980 Census & 1990 Estimates/1980 Census \\
Total   Absolute Difference & 5,604,178        & 29,016,212    & 0.1931                                                                             \\
Index of   Dissimilarity    & 0.011          & 0.054       & 0.2024                                                                             \\
$\chi^2$            & 243,515         & 5,808,180     & 0.0419                                                                             \\
Pearson's $\chi^2$ divergence                         & 0.0009         & 0.0186      & 0.0499    
\end{tabular}
\end{table}
The percent difference between their ratio and that of the Total Absolute Difference is about 5\%.  The percentage difference between $\chi^2$  and Pearson's $\chi^2$ divergence  is much greater: about 20\%.  It is not clear why these numbers are so different.  Exchangeability, the more favorable case of Proposition \ref{prop:cn_speed}, only guarantees convergence in proportion to $\sqrt{n}$. It is silent about the level.  Different datasets may produce different results.

\section{Conclusion\label{conclusion}}

We have shown that the weighted exponentiated loss functions are asymptotically equivalent whether levels or shares are used as arguments.  This has the important implication that numerical and distributive accuracy are asymptotically equivalent concepts.  Thus, asymptotically, it makes no difference whether shares or levels are used as the arguments in the loss functions.  In particular, the index of dissimilarity has been shown to be asymptotically equivalent to the total and mean absolute difference, thereby limiting its usefulness.  Similar results can be found for other difference measures.

\section*{References}

Armstrong, J. Scott (1985), Long-Range Forecasting: From Crystal Ball to Computer, New York: Wiley.\newline
\newline
Cohen, Michael L. and Xiao Di Zhang (1988), “Aggregate and Proportional Loss Functions in Adjustment using Artificial Populations,” SRD Research Report CENSUS/SRD/RR-88/11.  \url{http://www.census.gov/srd/papers/pdf/rr88-11.pdf}\newline
\newline
Coleman, Charles D. (2025), "Total Total Loss Functions for Measuring the Accuracy of Nonnegative Cross-Sectional Predictions," \url{https://doi.org/10.48550/arXiv.2507.15136}.\newline
\newline
Davis, Sam T. (1994), “Evaluation of Postcensal County Estimates for the 1980s,” Population Division Working Paper No. 5, U.S. Census Bureau, Washington, DC.  \url{http://www.census.gov/population/www/documentation/twps0005/twps0005.html}\newline
\newline
Kuha, Jouni and David Firth (2011),
"On the index of dissimilarity for lack of fit in loglinear and log-multiplicative models,"
Computational Statistics \& Data Analysis,
55(1),
375-388. \url{https://doi.org/10.1016/j.csda.2010.05.005}\newline
\newline
White, Michael J. (1986), “Segregation and Diversity Measures in Population Distribution,” Population Index, 52, 198-221.

\end{document}